\documentclass[10pt,a4paper]{article}
\usepackage{hyperref}

\usepackage{a4wide}
\usepackage{graphicx,xcolor,textpos}
\usepackage{helvet}
\usepackage[utf8]{inputenc}
\usepackage[T1]{fontenc}
\usepackage[english]{babel}
\usepackage{amsmath}
\usepackage{amsfonts}
\usepackage{amssymb,amsthm}
\usepackage{graphicx}
\usepackage{mathtools}
\usepackage{bbold}
\usepackage{comment}
\usepackage{csquotes}
\usepackage{tikz-cd}
\usepackage{cases}
\usepackage{listings}
\usepackage{algorithm}
\usepackage{algorithmic}
\newtheorem{thm}{Theorem}[section]

\newtheorem{lemma}[thm]{Lemma}
\newtheorem{prop}[thm]{Proposition}

\newtheorem{cor}[thm]{Corollary}
\theoremstyle{definition}
\newtheorem{df}[thm]{Definition}

\newtheorem{ques}{Question}
\theoremstyle{remark}
\newtheorem*{rem}{Remark}
\theoremstyle{definition}
\newtheorem{ex}[thm]{Example}
\theoremstyle{definition}

\newcommand{\Z}{\mathbb{Z}}
\newcommand{\N}{\mathbb{N}}
\newcommand{\Q}{\mathbb{Q}}
\newcommand{\R}{\mathbb{R}}
\newcommand{\C}{\mathbb{C}}
\newcommand{\norm}[1]{\|#1\|}
\renewcommand{\Im}{\text{Im\,}}
\DeclareMathOperator{\Aut}{Aut}
\DeclareMathOperator{\Inn}{Inn}
\DeclareMathOperator{\Out}{Out}
\DeclareMathOperator{\Inv}{Inv}
\DeclareMathOperator{\Com}{Com}
\DeclareMathOperator{\Gal}{Gal}
\DeclareMathOperator{\lcm}{lcm}
\DeclareMathOperator{\ggd}{gcd}
\DeclareMathOperator{\GL}{GL}
\DeclareMathOperator{\RF}{RF}

\DeclareMathOperator{\Id}{Id}

\title{Residual Finiteness Growth in Virtually Abelian Groups}
\author{Jonas Der\'e and Joren Matthys\thanks{KU Leuven Campus Kulak Kortrijk, Department of Mathematics, Research unit `Algebraic Topology and Group Theory', B-8560 Kortrijk, Belgium. The authors were supported by Internal Funds KU Leuven (project number 3E220559).}}
\date{}
\begin{document}
	\maketitle
	\begin{abstract}
	A group $G$ is called residually finite if for every non-trivial element $g \in G$, there exists a finite quotient $Q$ of $G$ such that the element $g$ is non-trivial in the quotient as well. Instead of just investigating whether a group satisfies this property, a new perspective is to quantify residual finiteness by studying the minimal size of the finite quotient $Q$ depending on the complexity of the element $g$, for example by using the word norm $\norm{g}_G$ if the group $G$ is assumed to be finitely generated. The residual finiteness growth $\RF_G: \N \to \N$ is then defined as the smallest function such that if $\norm{g}_G \leq r$, there exists a morphism $\varphi: G \to Q$ to a finite group $Q$ with $|Q| \leq \RF_G(r)$ and $\varphi(g) \neq e_Q$. 
	
	Although upper bounds have been established for several classes of groups, exact asymptotics for the function $\RF_G$ are only known for very few groups such as abelian groups, the Grigorchuk group and certain arithmetic groups. In this paper, we show that the residual finiteness growth of virtually abelian groups equals $\log^k$ for some $k \in \N$, where the value $k$ is given by an explicit expression. As an application, we show that for every $m \geq 1$ and every $1 \leq k \leq m$, there exists a group $G$ containing a normal abelian subgroup of rank $m$ and with $\RF_G \approx \log^k$.
	\end{abstract}
	\section{Introduction}
In a residually finite group $G$ there exists by definition for every non-trivial element $g\in G$ a group morphism $\varphi: G \to Q$ to a finite group $Q$ such that $\varphi(g)$ is still non-trivial. Examples of such groups include finite groups, free groups, finitely generated virtually nilpotent and more generally polycyclic groups, finitely generated linear groups, and fundamental groups of compact 3-manifolds. As every residually finite group is Hopfian, the Baumslag-Solitar group $BS(2,3)$ on the other hand is a non-example.

Quite recently, the quantification of this property for finitely generated groups $G$ was initiated by Bou-Rabee in his paper \cite{bou2010quantifying}. More precisely, Bou-Rabee studied the asymptotic behavior of the (normal) residual finiteness growth $\RF_G: \N \to \N$, i.e. the minimal function such that if $\norm{g}_G \leq r$, then $Q$ exists as above with $|Q| \leq \RF_G(r)$. Here $\norm{\cdot}_G$ denotes a fixed word norm on $G$, induced by a finite generating set $S$, so satisfying $\norm{g}_G \leq r$ if and only if $g$ can be written as a product of at most $r$ elements in $S \cup S^{-1}$. Although needed for the definition, the exact generating set $S$ does not play a role when considering $\RF_G$ up to a certain equivalence relation, see Section \ref{sec_resid_fin_growth} for more details.
	
	Since this first paper in 2010, the group invariant $\RF_G$ has been studied for different classes of residually finite groups. In several cases upper bounds have been found, for example in the case of linear groups in \cite{franz2017quantifying}, $S$-arithmetic subgroups of higher rank Chevalley groups in \cite{bou2012quantifying}, the first Grigorchuk group \cite{bou2010quantifying}, free groups \cite{bradford2019short}, lamplighter groups \cite{bou2019residual}, \dots~However, exact bounds are rare in the literature, most strikingly illustrated by the fact that the exact asymptotic behavior for free groups is still unknown. 
	
	In the same light, the exact connection between the residual finiteness growth of groups and certain group constructions remains largely unknown, for instance when considering its automorphism group, forming semi-direct products or wreath products,\ldots \ In the case of taking finitely generated subgroups $K$ of $G$, it has been shown in \cite{bou2010quantifying} that $\RF_K$ is bounded above by $\RF_G$. Furthermore, if $K$ is a finite extension of $G$ with index $n = [G:K]$, then $\RF_G$ is itself bounded above by $(\RF_K)^n$. We refer to the survey article \cite{survey2022} for a more detailed discussion about the known results and the relation to other residual properties such as conjugacy separability.  
	
An interesting result in this context is the following from \cite{bou2011asymptotic}. In this paper, the authors characterize the virtually nilpotent groups within the class of linear groups as exactly those group $G$ for which $\RF_G$ is bounded by $\log^k$ for some $k\geq 0$. This raises the following question with respect to the exact asymptotics of $\RF_G$ for these groups.
	\begin{ques}
		\label{Q1}
	Let $G$ be a finitely generated virtually nilpotent group, does there exist a $k \geq 0$ with $\RF_G$ equal to $\log^k$?
	\end{ques}
Note that the residual finiteness growth is bounded if and only if the group is finite, and if the group contains an infinite cyclic group, then $\log \preceq \RF_G$. So far, Question \ref{Q1} has only been answered for abelian groups, where $\RF_G$ is either bounded (so $\log^0$) for the finite ones or $\log$ for the infinite groups respectively. The answer is also positive for all examples of nilpotent groups for which $\RF_{G}$ has been computed, as for instance the Heisenberg group for which $\RF_G$ equals $\log^3$. In this paper, we give a positive answer to Question \ref{Q1} for all virtually abelian groups. 
	\begin{thm}
		Let $G$ be a finitely generated virtually abelian group, then $\RF_G$ equals $\log^k$ for some $k\geq 0$.
	\end{thm}
Moreover, this $k$ has an explicit form, depending on an induced integral representation. Indeed, take any abelian torsion-free normal subgroup $K \triangleleft~G$ of finite index, then the finite group $H = G/K$ acts via conjugation on $K$, leading to a representation $\varphi: H \to \Aut(K) \cong \GL(m,\Z)$ with $m$ the rank of $K$. In the case of crystallographic groups, so discrete and cocompact subgroups of Euclidean isometries, the representation $\varphi$ is equal to the holonomy representation, getting its name from its geometric interpretation.

Over the complex numbers $\mathbb{C}$, the representation $\varphi$ decomposes into irreducible subrepesentations of dimensions $\leq m$. We establish the following refinement of the previous theorem:
	\begin{thm} \label{thm_main}
		Let $G$ be a virtually abelian group with torsion-free abelian subgroup of rank $m$. Then $\RF_G$ equals $\log^k$, where $0 \leq k \leq m$ is the maximal dimension of the irreducible subrepresentations of $\varphi$ over $\mathbb{C}$.
	\end{thm}
Moreover, for every $m\geq 1$ we give an example to show that every $1 \leq k \leq m$ is possible. This result can also be interpreted as a specialization of the general bounds for finitely generated subgroups as described above. Indeed, the group $G$ is a finite extension of an abelian subgroup $K$ of rank $m\geq 1$ and hence the previous results asserted that $\RF_G$ was bounded below and above by $\log$ and $\log^n$ respectively, with $n = [G:K]$. However, now we know it equals $\log^k$ for some $1\leq k \leq \min\{m,n\}$, with various examples where the general bounds are not optimal.

The outline of this paper is as follows. In section \ref{sec_resid_fin_growth}, we will introduce the notion of residual finiteness growth with respect to families of normal subgroups, as will be crucial later. Next, in section \ref{sec_virt_ab}, we will define virtually abelian groups, and we will make some preliminary considerations concerning their residual finiteness growth. Following this initial discussion, section \ref{sec_upper} and \ref{sec_lower} will focus on the upper and the lower bound of the main result in theorem \ref{thm_main}. In order to do so, section \ref{sec_upper} will give more background about linear representations and section \ref{sec_lower} about matrices commuting with such a representation. In the last subsection of section \ref{sec_lower}, the main result will then be proven. Lastly, we will conclude this article with some applications and open questions in sections \ref{sec_applic} and \ref{sec_open_questions}.
	
	\section{Residual finiteness growth} \label{sec_resid_fin_growth}
	
	In this subsection, we introduce the concept of residual finiteness growth and we prove some of its properties. However, as this will be important for future results, we will work in the more general context of an arbitrary family $P$ of normal subgroup in $G$, in a similar flavor as \cite{bou2015residual}. 
	
	Every group under consideration is assumed to be finitely generated and residually finite. For these groups, we have a word norm $\norm{\cdot}_G$ on $G$, which measures the number of generators needed to write an element of $G$. As the only information we need from our metric are the balls of a certain radius, we will include this information in our definitions.
		\begin{df}
		For a group $G$ with finite generating set $S$, we define the ball $B_G(r)$ of radius $r > 0$ as 
		\begin{align*}B_G(r) &= \left\{g\in G \mid \norm{g}_G \leq r	\right\} \\ &= \left\{s_1^{\pm1} \ldots s_k^{\pm1} \mid s_i \in S, k \leq r	\right\} \end{align*}
	\end{df}
\noindent As we will introduce in Definition \ref{df_resid_fin}, we will define residual finiteness growth in such a way that it is independent of the choice of word metric. However, for the time being, we keep the balls $B_G$ in our definitions.
	
	\begin{df}
		We write $\nu(G) \subset \mathcal{P}(G)$ for the set of all finite index, normal subgroups of $G$. 
	\end{df}
\noindent Note that a group $G$ is residually finite if and only if $\displaystyle \bigcap_{N\in \nu(G)} N = \left\{e_G\right\}$. However, for the remainder it will be important to build the theory for more general subsets $P \subset \nu(G)$.
	\begin{df}
		Let $P$ be a subset of $\nu(G)$ for which $\displaystyle \bigcap_{N\in P}N = \left\{e_G\right\}$.
		Define the divisibility function $D_{G,P}: G \to \N \cup \left\{\infty
\right\}$ with respect to $G$ and $P$ as 
		\begin{equation*}
			D_{G,P}(g) =  \begin{cases} \min\left\{[G:N]\mid g\notin N, N \in P

\right\} & \text{for } g \neq e_G,\\ \infty & \text{for } g = e_G. \end{cases}
		\end{equation*}
	and the residual finiteness growth $\RF_{G,P,B_G}: \mathbb{R}_{\geq 1} \to \mathbb{N}$ with respect to $G$, $P$ and $B_G$ as
		\begin{equation*}
			\RF_{G,P,B_G}(r)=\max\left\{D_{G,P}(g)\mid e_G \neq g \in B_G(r) \right\}.
		\end{equation*}
	\end{df}
By definition we let $D_{G,P}(e_G) = \infty$, as this will be convenient for writing certain upper bounds as in Lemma \ref{lem_direct_sum}. From now on, we will always assume that our subsets $P \subset \nu(G)$ satisfy $\displaystyle \bigcap_{N\in P}N = \left\{e_G\right\}$. In particular, the set $P$ is always non-empty.

We will study the function $\RF_{G,P,B_G}$ up to the following preorder and corresponding equivalence relation:
	\begin{df}
		Let $f,g: \mathbb{R}_{\geq 1} \to \mathbb{R}_{\geq 1}$ be increasing functions. We define the following relations:
		\begin{align*}f &\preceq g \Leftrightarrow \exists C >0: \forall r \geq \max\{1,1/C\}: f(r) \leq Cg(Cr)\\
		f&\approx g \Leftrightarrow f\preceq g \text{ and } g \preceq f\end{align*}
	\end{df}
	We will give some lemmas concerning the behavior of $RF_{G,P, B_G}$ and $D_{G,P}$ with respect to changing $G$, $P$ or the metric balls $B_G$. It should be noted that some of the results in this section still hold in a more general setting, for example when allowing for normal subgroups of infinite index in $\nu(G)$ or when relaxing the condition that $\displaystyle \cap_{N\in P}N = \left\{e_G \right\}$. 
	
We introduce the following operations on subsets $P \subset \nu(G)$.
	\begin{df}
		\begin{itemize}
		\item Let $K\leq G$ and $P \subset \nu(G)$. We define the intersection as
		$$P\cap K = \left\{N\cap K \mid N \in P \right\} \subset \nu(K) .$$
	Note that, if $K$ is a normal subgroup of $G$, every $N \cap K$ is normal in $G$ as well, and thus we can also consider $P \cap K$ as a subset of $\nu(G)$. It will be clear from the context which convention we use. 	
	\item Let $P_1\subset \nu(G_1)$ and $P_2 \subset \nu(G_2)$. The direct sum is defined as
		$$P_1\oplus P_2 = \left\{N_1\oplus N_2\mid N_1\subset P_1, N_2\subset P_2 \right\} \subset \nu(G_1\oplus G_2).$$
		\item Let $\psi: G_1 \to G_2$ be a  homomorphism and $P\subset \nu(G_2)$. We define the inverse image of $P$ as
		$$\psi^{-1}(P) = \left\{\psi^{-1}(N)\mid N\subset P \right\} \subset \nu(G_1).$$
		Note that the inverse image of a normal subgroup of finite index is again normal of finite index.
		\end{itemize}
	\end{df}
\noindent It should be noted that the first two constructions above do not change the assumption on  $\displaystyle \bigcap_{N\in P}N = \left\{e_G\right\}$, as can be checked by the reader. For example, in the first case, we have that $$\displaystyle \bigcap_{N\in P \cap K}N =  \left(\bigcap_{N\in P}N \right) \cap K=  \left\{e_G\right\} \cap K = \left\{e_G\right\}.$$ However, the last construction does not preserve this property, as the kernel of $\psi$ always lies in $\psi^{-1}(N)$. We will only use the latter in Lemma \ref{lem_surjective} for the divisibility function.

	\begin{lemma} \label{lem_change_P}
		Let $G$ be a group and $P, \bar{P}$ be subsets of $\nu(G)$. If there exists a constant $C>0$ such that for all $N\in P$, there exists $\bar{N} \in \bar{P}$ with $\bar{N} \subset N$ and $[G:\bar{N}] \leq C[G:N]$, then $\RF_{G, \bar{P}, B_G} \preceq \RF_{G,P,B_G}$. 
	\end{lemma}
	\begin{proof}
		Let $e_G \neq g \in B_G(r)$ and take $N\in P$ such that $g\notin N$ and $[G:N] = D_{G,P}(g)$. Take $\bar{N} \subset N$ as in the condition of the lemma, then we know that $g\notin \bar{N}$ and hence $$D_{G, \bar{P}}(g) \leq [G: \bar{N}] \leq C[G:N] = C\cdot D_{G,P}(g) \leq C\cdot\RF_{G,P,B_G}(r).$$ Taking the maximum over all non-trivial $g$ in $B_G(r)$ gives the result. 
	\end{proof}

A particular case of this lemma will be used several times in this paper:

\begin{ex}
	\label{ex:sub}
For any group $G$ and $P \subset \bar{P} \subset \nu(G)$ with $\displaystyle \bigcap_{N\in P}N = \left\{e_G\right\}$, the conditions of Lemma \ref{lem_change_P} are automatically satisfied by taking  $N = \bar{N}$ and $C=1$. We conclude that $\RF_{G,\bar{P},B_G} \preceq \RF_{G,P,B_G}$ in this case.
\end{ex}

	Note that if $K$ is a finitely generated subgroup of $G$, then there always exists a constant $D>0$ such that $B_K(D\cdot r) \subset B_G(r)$. If $K$ has finite index in $G$, then $K$ is automatically finitely generated and moreover there exists a constant $C \geq 0$ such that $B_G(r)\cap K \subset B_K(C\cdot r)$, see \cite[corollary 5.4.5]{loh2017geometric}.
	\begin{lemma} \label{lem_reduce_group}
		Let $P\subset \nu(G)$ and suppose $K$ is a finite index subgroup of $G$, then $\RF_{K,P\cap K, B_K} \preceq \RF_{G,P,B_G}$. If moreover every $N\in P$ satisfies $N \subset K$ then $\RF_{K,P\cap K, B_K} \approx \RF_{G,P,B_G}$.
	\end{lemma}

The last condition means that we can consider $P \subset \nu(K)$ as well, as every element lies entirely in $K$. Of course, not every element of $\nu(K)$ lies in $\nu(G)$. 

	\begin{proof}
Take $D>0$ such that $B_K(D\cdot r) \subset B_G(r)$ and $e_K \neq k \in B_K(D\cdot r)$. Since $e_G \neq k \in B_G(r)$ we find $N\in P$ such that $k\notin N$ and $[G:N] = D_{G,P}(k)$. Hence we have 
		$$D_{K,P\cap K}(k) \leq [K:K\cap N] \leq [G:N] = D_{G,P}(k) \leq \RF_{G,P,B_G}(r),$$
		using that $[G:N] = [G:KN]\cdot [K:K\cap N]$. Taking the maximum over all non-trivial $k$ in $B_K(D\cdot r)$ gives the inequality $\RF_{K, P\cap K, B_K} \preceq \RF_{G,P,B_G}$.
		
		For the second statement, take $C$ such that $B_G(r)\cap K \subset B_K(C\cdot r)$. Fix a normal subgroup $N'$ of finite index in $G$, lying in $K$, which exists by our assumptions on $P$. Now take $e_G \neq g \in B_G(r)$. If $g \notin K$, then $D_{G,P}(g) \leq [G:N']$. If $g \in K$ and hence $k\in B_K(C\cdot r)$, we can take $N \in P$ such that $g \notin N$ and $[K:N] = D_{K,P}(g)$. We obtain
		\begin{align*}
			D_{G,P}(g) & \leq [G:N] = [G:K]\cdot [K:N] \\ & = [G:K]\cdot D_{K,P}(g) \leq [G:K]\cdot \RF_{K,P,B_K}(C\cdot r).
		\end{align*}
		In total, we argued that $D_{G,P}(g) \leq \max\left\{[G:N'], [G:K]\cdot\RF_{K,P,B_K}(C\cdot r) \right\}$. As $\RF_{K,P,B_K}(C\cdot r)$ is at least a bounded function, taking the maximum over all $e_G \neq g\in B_G(r)$ leads to the desired result.
	\end{proof}

As a consequence, we see that the finite generating set does not play a role when studying the residual finiteness growth. 
	\begin{cor} \label{cor_independent}
		Let $B_G$ and $B_{G}'$ be two metric balls associated to different word metrics on $G$. Then $\RF_{G,P,B_G} \approx \RF_{G,P,B_{G}'}$.
	\end{cor}
\begin{proof}
	Apply the previous lemma to the case where $K = G$ and $B_K = B_{G}'$.
\end{proof}
This way we can define residual finiteness growth on $G$ independent of the word metric. For the sake of this paper, the metric will play no active role. Therefore, we will no longer specify the word metrics in our results.
\begin{df} \label{df_resid_fin}
	The residual finiteness growth $\RF_G$ of a group $G$ is defined as the equivalence class of $\RF_{G, \nu(G),B_G}$. We will say that $\RF_G$ equals $f$ if $\RF_{G, \nu(G),B_G} \approx f$ for some (and hence every) finite generating set $S$.
\end{df}

The next lemma forms a combination of Lemmata \ref{lem_change_P} and \ref{lem_reduce_group}.
\begin{lemma} \label{lem_reduc_two}
	Suppose $K$ is a normal, finite index subgroup of $G$. Let $P\subset \nu(G)$ with $G\in P$ and $\bar{P} \subset P\cap K \subset \nu(K)$.	If there exists a constant $C>0$ such that for all $N\in P$ there exists $\bar{N} \in \bar{P}$ such that $\bar{N}\subset N$ and $[K:\bar{N}] \leq C[G:N]$, then $\RF_{G,P,B_G} \approx \RF_{K, \bar{P},B_K}$.
\end{lemma}
\begin{proof}
	We will show the following equivalences:
	\begin{equation*}
		\begin{split}
			\RF_{G,P,B_G} & \approx \RF_{G,P\cap K,B_G} \\
			& \approx \RF_{K,P\cap K,B_K} \\
			& \approx \RF_{K, \bar{P},B_K}.
		\end{split}
	\end{equation*}
For the first equivalence, we first show that $\RF_{G,P,B_G} \preceq \RF_{G,P\cap K,B_G}$. Take $g\in B_G(r)$ arbitrary, then either $g\notin K$ or $g\in K$. If $g\notin K$, $D_{G,P}(g) \leq [G:G\cap K] = [G:K]$. In the other case, take $g\notin N\cap K$ that realizes $D_{G,P\cap K}(g)$, i.e. $[G:N\cap K] = D_{G,P\cap K}(g)$. Then it also holds that $g\notin N$. Hence, we see that
$$D_{G,P}(g) \leq [G:N] \leq [G:N\cap K] = D_{G,P\cap K}(g) \leq \RF_{G,P\cap K,B_G}(r).$$
In conclusion, $D_{G,P}(g) \leq \max\left\{[G:K], \RF_{G,P\cap K,B_G} \right\}$. Taking the maximum over all $g\in B_G(r)$ gives the result.

Conversely, if $N\lhd G$, we have $[G:N] = [G:NK]\cdot [K:N\cap K]$, so $[G:N] \geq [K:N\cap K]$. As a consequence,
$$[G:N\cap K] = [G:K]\cdot [K:N\cap K] \leq [G:K]\cdot [G:N].$$
By lemma \ref{lem_change_P} with $C = [G:K]$, we conclude the first equality holds.

The second equivalence is a direct application of Lemma \ref{lem_reduce_group}.

For the third equivalence, we note that by Example \ref{ex:sub} it holds that $\RF_{K,P\cap K,B_K} \preceq \RF_{K, \bar{P},B_K}$. For the converse, we wish to apply Lemma \ref{lem_change_P}. To do so, we claim that for every $N\cap K \in P\cap K$, there exists $\bar{N} \in \overline{P} $ with $\bar{N}\subset N\cap K$ and $[G:\bar{N}] \leq C[G:K] \cdot [K:N\cap K]$. Here, the constant in the lemma's statement is in fact $C[G:K]$. Take $N\cap K \in N\cap P$. By assumption, there exists $\bar{N} \in \bar{P}$ such that $\bar{N} \subset N$ and $[K: \bar{N}] \leq C[G:N]$. This implies that $\bar{N} \subset N\cap K \in P\cap K$ as $\bar{P} \subset P\cap K$ and 
$$[K: \bar{N}] \leq C[G:KN]\cdot [K:N\cap K] \leq C[G:K] \cdot [K:N\cap K].$$
This shows the claim and therefore ends the proof.
\end{proof}
We now look at how residual finiteness behaves with respect to direct sums.
	\begin{lemma} \label{lem_direct_sum}
	Let $G = G_1\oplus G_2$ be a direct sum of two groups $G_i$ and $P \subset \nu(G)$. If we write $P_i = P \cap G_i$ and assume that $\bar{P} = P_1\oplus P_2 \subset P$, $G_1 \in P_1$ and $G_2 \in P_2$, then
		\begin{equation} \label{eq_minimum}
		D_{G,P}(g) \leq \min\left\{D_{G_1,P_1}(\pi_1(g)), D_{G_2,P_2}(\pi_2(g)) \right\},
		\end{equation}
		and 
		$$\RF_{G,P,B_G} \approx \max\left\{\RF_{G_1,P_1,B_1}, \RF_{G_2,P_2,B_2} \right\}.$$
		The maps $\pi_1$ and $\pi_2$ are the natural projections from $G$ onto $G_1$ and $G_2$.
	\end{lemma}
	\begin{proof}
		Suppose $G = G_1\oplus G_2$ and $\bar{P} = P_1\oplus P_2 \subset P$. We clearly have $D_{G,P}(g) \leq D_{G, \bar{P}}(g)$ by Example \ref{ex:sub}. If $N = N_1\oplus N_2 \in \bar{P}$ is the normal subgroup such that $g \notin N$ and $[G:N] = D_{G, \bar{P}}(g)$, then either $\pi_1(g) \notin N_1$ or $\pi_2(g) \notin N_2$. Suppose that $\pi_1(g) \notin N_1$, then $g\notin N_1\oplus G_2 \in \bar{P}$ and thus $D_{G, \bar{P}}(g) = [G_1:N_1]$. Choosing $N_1$ to have minimal index in $G_1$ with $\pi_1(g) \notin N_1$ shows that $D_{G, \bar{P}}(g) \leq D_{G_1,P_1}(\pi_1(g))$. Analogously if $\pi_2(g) \notin N_2$, we get $D_{G, \bar{P}}(g) \leq D_{G_2,P_2}(\pi_2(g))$. From this, we conclude that 
		$$D_{G,P}(g) \leq \min\left\{D_{G_1,P_1}(\pi_1(g)), D_{G_2,P_2}(\pi_2(g)) \right\}.$$
		
		Now, we need to show that 
		$$\RF_{G,P,B_G} \approx \max\left\{\RF_{G_1,P_1,B_1}, \RF_{G_2,P_2,B_2} \right\}.$$	Since the definition of $\RF$ is up to equivalence independent of the choice of metric, we take $B_G(r)$ as $$B_G(r) = \left\{g\mid \pi_1(g)\in B_{G_1}(r_1), \pi_2(g)\in B_{G_2}(r_2), r_1+r_2 \leq r \right\}.$$ These are exactly the metric balls arising from adjoining generators $S_1$ of $G_1$ with $S_2$ of $G_2$, more concretely with generating set $S = \{(s,e_{G_2}) \mid s \in S_1\} \cup \{(e_{G_1},s) \mid s \in S_2\}$ for $G_1 \oplus G_2$.
		
		From one side, we have for all $e_G\neq g\in B_G(r)$ that either $\pi_1(g)$ or $\pi_2(g)$ is non-trivial. By symmetry, it suffices to consider the case when $\pi_1(g) \neq e_{G_1}$. Since $\pi_1(g) \in B_{G_1}(r)$, we find that 
		\begin{align*}
			D_{G,P}(g) & \leq  \min\left\{D_{G_1,P_1}(\pi_1(g)), D_{G_2,P_2}(\pi_2(g))\right\}\\ & \leq D_{G_1,P_1}(\pi_1(g)) \\ & \leq \RF_{G_1,P_1,B_1}(r) \\& \leq  \max\left\{\RF_{G_1,P_1,B_1}(r), \RF_{G_2,P_2,B_2}(r)\right\}.	\end{align*}
since $ D_{G_1,P_1}(\pi_1(g)) < \infty$.	This shows the first inequality.

		From the other side, take $g\in B_1(r) \subset G_1$ such that $\RF_{G_1,P_1,B_1}(r) = D_{G_1,P_1}(g)$. Then for all $N\in P$ with $g\notin N\cap G_1 \in P_1$ we have 
		$$[G:N]\geq [G_1:N\cap G_1] \geq D_{G_1,P_1}(g) = \RF_{G_1,P_1,B_1}(r),$$
		or after taking the minimum over all such normal subgroups:
		$$D_{G,P}(g) \geq D_{G_1,P_1}(g) = \RF_{G_1,P_1,B_1}(r).$$
		Hence surely, $\RF_{G,P,B_G}(r) \geq  \RF_{G_1,P_1,B_1}(r)$. By symmetry, this argument also applies to $G_2$, therefore ending the proof.
	\end{proof}

	Note that the conditions $G_1 \in P_1$ and $G_2 \in P_2$ make sure that the inequality \eqref{eq_minimum} holds. If those conditions are removed but $P$ is non-empty, a multiplicative constant depending on $G$ and $P$ but not on $g$ should be added.

	\begin{lemma} \label{lem_surjective}
		If $\psi: G_1 \to G_2$ is a surjective homomorphism, $P\subset \nu(G_2)$ and $\psi^{-1}(P) \subset \bar{P}$, then $D_{G_1, \bar{P}}(g) \leq D_{G_2, P}(\psi(g))$.
	\end{lemma}
	\begin{proof}
	Since $\psi^{-1}(P) \subset \bar{P}$, it is immediate that
		$$D_{G_1, \bar{P}}(g) \leq D_{G_1, \psi^{-1}(P)}(g).$$
		We will now show that
		$$D_{G_1, \psi^{-1}(P)}(g) \leq D_{G_2, P}(\psi(g)).$$ If $ D_{G_2, P}(\psi(g)) = \infty$, then there is nothing to prove, so we can assume that $ D_{G_2, P}(\psi(g)) < \infty$.
Choose $N \in P$ such that $\psi(g) \notin N$ and $[G_2:N] = D_{G_2, P}(\psi(g))$, then $g\notin \psi^{-1}(N)$ and so
		\begin{align*}
			D_{G_1, \psi^{-1}(P)}(g) & \leq [G_1:\psi^{-1}(N)] \\  & = [G_2:N] \\ &= D_{G_2, \psi(P)}(\psi(g)).
		\end{align*}	\end{proof}
	In particular, the previous lemma applies to the case where $\bar{P} = \nu(G_1)$ and $P = \nu(G_2)$.

	\section{Residual finiteness growth of virtually abelian groups} \label{sec_virt_ab}
	
	In this subsection, we will find a new expression for the residual finiteness growth of virtually abelian groups, using results of the previous sections for families $P$ of normal subgroups.
	
	Let us first start with a general group extension $G$, namely a group fitting in a short exact sequence
	$$\begin{tikzcd}
		1 \arrow[r] & K \arrow[r, "i", hook] & G \arrow[r, "\pi", two heads] & H \arrow[r] & 1. \label{eq_short_exact}
	\end{tikzcd}$$
If $s: H\to G$ is a set theoretic map with the property that $\pi\circ s = \Id$, called a section, then \begin{align*}\varphi: H &\to \Out(K)\\ h  & \mapsto (i^{-1}\circ C_{s(h)}\circ i)\Inn(K)
		\end{align*} is a well-defined homomorphism, independent of the chosen section. Here $C_g: K \to K$ for $g \in G$ stands for the automorphism of $K$ given by conjugation with $g$, so $C_g(k) = g k g^{-1}$. We write $\Inn(G) = \{C_g \mid g \in G\} \subset \Aut(G)$ for the normal subgroup which contains every conjugation automorphism, and $\Out(G) = \Aut(G) / \Inn(G)$ for the quotient group of outer automorphisms.
	
	A first result in this section is Theorem \ref{thm_finite_extension} below, which relates $\RF_G$ to $\RF_{K,P}$ for a certain family of subgroups $P$, namely the ones invariant under the morphism $\varphi: H \to \Out(K)$. We start with the following observation.
	
	\begin{lemma}
	Assume that $\psi_1$ and $\psi_2$ are two automorphisms of $K$ such that $\psi_1 \Inn(K)= \psi_2 \Inn(K)$, then $\psi_1(N) = \psi_2(N)$ for any normal subgroup $N$ of $K$.
	\end{lemma}  
\begin{proof}
By assumption, there exists $k \in K$ such that $\psi_1 = \psi_2 \circ C_k$. For any normal subgroup $N \triangleleft K$ we have that $$\psi_1(N) = \psi_2(C_k(N)) = \psi_2(N).$$
\end{proof}

\noindent Hence, this lemma shows that we can define $\overline{\psi}(N)$ for any normal subgroup and any element $\overline{\psi} \in \Out(K)$, making sense of the following definition.
	\begin{df}
		Let $\varphi: H \to \Out(K)$ be a morphism, then we define $\Inv(\varphi) \subset \nu(K)$ as the set $$\Inv(\varphi) = \left\{N\in \nu(K)\mid \forall h \in H:  \varphi(h)(N) = N \right\}.$$
	\end{df}

The importance of this set lies in the following application of Lemma \ref{lem_reduc_two}.
	%	Although we give a direct proof, the theorem below can also be seen as an application of proposition 2.3 of \cite{bou2015residual}.
	\begin{thm} \label{thm_finite_extension}
		Let $G$ be a residually finite, finitely generated group in a short exact sequence as in equation \eqref{eq_short_exact}.
		If $H$ is finite, then $\RF_{G, \nu(G), B_G}$ equals $\RF_{K, \Inv(\varphi), B_K}$.
	\end{thm}
	\begin{proof}
		We first claim that $\nu(G)\cap K$ equals precisely $\Inv(\varphi)$. Indeed, take $N$ normal in $G$ arbitrary. Since $K$ is also a normal subgroup, the intersection $\bar{N} := N\cap K$ is normal in $G$. Take any $h\in H$, then $C_{s(h)}(\bar{N}) = \bar{N}$, where $C_{s(h)}$ denotes conjugation as mentioned above. By definition, this is precisely $\varphi(h)(\bar{N})$, so $\bar{N} \in \Inv(\varphi)$. 
		
		Conversely, if $N \in \Inv(\varphi)$, then $N\leq K$ and for $g\in G$ arbitrary we have
		$$g^{-1}Ng = s(h)^{-1}k^{-1}Nks(h) = \varphi(h)(N) = N,$$
		by writing $g = ks(h)$ for $k\in K$ and $h\in H$, so $N\in \nu(G)\cap K$.
		
		The statement of the theorem is now a direct application of Lemma \ref{lem_reduc_two}, because $K\lhd_f G$ and $\nu(G)\cap K$ equals precisely $\Inv(\varphi)$. For every $N\lhd G$, we can now set $\bar{N} = N\cap K$. We have the inequality $[K:\bar{N}] \leq [G:N]$, because $[G:N] = [G:KN]\cdot [K:K\cap N]$.
	\end{proof}
	Note that the previous result holds for general groups $K$, without assuming it is abelian as we will do further on. However, the assumption that $H$ is finite is crucial in order to use Lemma \ref{lem_reduc_two}, as we illustrate with the following example.
	
	\begin{ex}
		The groups $\Z^3$ and the discrete Heisenberg group $H_3(\Z)$ have different residual finiteness growth, namely $\log$ and $\log^3$ respectively, see \cite{bou2010quantifying}. However, both groups have a normal subgroup $\Z$ with quotient $\Z^2$, namely $$ \dfrac{\Z^3}{\Z\times\left\{0 \right\}^2} \cong \Z^2 \cong \dfrac{H_3(\Z)}{Z(H_3(\Z))},$$ and in both cases the morphism $\varphi: \Z^2 \to \Out(\Z) = \Aut(\Z)$ induced by conjugation is the trivial one.
	\end{ex}

	The class of virtually abelian groups is exactly the special case where $K$ is an abelian group and $H$ is finite. 
	\begin{df}
		A group $G$ is said to be virtually abelian if it has an abelian normal subgroup $K$ of finite index.
	\end{df}
Note that the map $\varphi$ is then a morphism $H\to \Aut(K)$, as $\Out(K) = \Aut(K)$ when $K$ is abelian. Since we assume that $G$ is finitely generated and hence $K$ as well, we know that, $K \cong \Z^m\oplus T$ for some $m\in \mathbb{N}$ and $|T| < \infty$ by the structure theorem of finitely generated abelian groups. If we take $K^\prime = \vert T \vert K \subset K$, then $K^\prime$ is a characteristic and torsion-free subgroup of $K$ and hence a normal subgroup of $G$. Hence without loss of generality we can assume that $K$ is a torsion-free normal subgroup. In fact, this argument is a simplified version of a more general statement for virtually polycyclic groups, which always have a finite index normal subgroup which is torsion-free, see \cite{ragh72}.

Since $K \cong \Z^m$ for some $m \geq 0$, we have reduced the problem of finding $\RF_G$ for a virtually abelian group $G$ to the question of determining $\RF_{\Z^m, \Inv(\varphi), B_{\Z^m}}$ for all $\varphi: H \to \GL(m,\Z)$ with $H$ finite. The next two sections deal with the upper and lower bound respectively. However, we first make some final comments about the representation $\varphi$ and how it depends on the choice of the subgroup $K$.

In the special case where $G$ is crystallographic, which is equivalent to $G$ being finitely generated, virtually abelian and for which every finite normal subgroup is trivial by \cite{deki96-1}, we can say more about the representation $\varphi$. Indeed, in this case, $G$ has a maximal abelian subgroup $A$ which is torsion-free and normal. In particular, we can take $K = A$, and then the representation $\varphi: H \to \Aut(A)$ becomes faithful. This representation is known as the holonomy representation. We refer to \cite{deki96-1} for more details. 

However, in the case where $G$ is not crystallographic, there is no canonical choice for the abelian normal subgroup $K$. For any two choices $K, K^\prime \subset G$ which are torsion-free abelian of maximal rank, we have that the groups $K$ and $K^\prime$ are commensurable, i.e.~$K \cap K^\prime$ is a subgroup of finite index in both $K$ and $K^\prime$. To see how the representation $\varphi$ varies over the subgroups $K$, it hence suffices to consider the case $K^\prime \subset K$ with corresponding groups $H = G / K, H^\prime = G /K^\prime$ and maps $\varphi: H \to \Aut(K)$, $\varphi^\prime: H^\prime \to \Aut(K^\prime)$. 

Note that the subgroup $K /K^\prime = M$ is a normal subgroup of $H^\prime$ which acts trivially on $K^\prime$ by conjugation, and thus $M$ lies in the kernel of $\varphi^\prime$. Hence, we find that the induced representation $$ \overline{\varphi}^\prime: H = H^\prime/M \to \Aut(K^\prime)$$ is by definition the restriction of the representation $\varphi$ to the invariant subgroup $K^\prime \subset K$. As $K^\prime$ has finite index in $K$, the representations $\varphi$ and $\varphi^\prime$ are equivalent over the rational numbers $\Q$. In particular, we conclude that the image $\varphi(H) \subset \GL(m,\Z)$ does not depend on the choice of $K$ up to $\Q$-equivalence. 

As our main result only depends on the representation $\varphi$ over $\C$, this $\Q$-equivalence does not play a role further and we can just fix one choice of subgroup $K$.

	\section{Proof of the upper bound} \label{sec_upper}
		
	In the previous section, we showed how $\RF_{G}$ for finitely generated virtually abelian group $G$ only depends on the corresponding representation $\varphi: H \to \GL(m,\Z)$ for some finite group $H$. 
	
	 Recall that a (linear) representation of a finite group $H$ is by definition a group homomorphism $\varphi : H \to \GL(m, \mathbb{F})$ for some field $\mathbb{F}$. The special case where the entries lie in $\Z$ is important for our main results.
	\begin{df}
		We say a group representation is integral if it is of the form $H\to \GL(m, \Z)$.
	\end{df}
If $\varphi: H \to GL(m,\mathbb{F})$ is a linear representation and $\mathbb{F} \subset \mathbb{E}$ is a field extension, we also get a linear representation $\varphi^\mathbb{E}: H \to GL(m,\mathbb{E})$. Sometimes we will write both representations as $\varphi$ if it is clear from the context over which field we are working.

The advantage of working over fields is that every representation can be decomposed into irreducible subrepresentations. If $\varphi: H \to \GL(m, \mathbb{F})$ is a representation, then a subrepresentation is a subspace $W \subset \mathbb{F}^m$ which is invariant under every element $\varphi(h)$ for $h \in H$. After choosing a basis for $W$, we then also get a map $H \to \GL(m^\prime,\mathbb{F})$. We call a representation irreducible if the only subrepresentations are the trivial ones, namely $0$ and $\mathbb{F}^m$. Every group representation $\varphi: H \to \GL(m, \mathbb{F})$ can be written as 
$$\varphi = \varphi_1 \times \dots \times \varphi_n,$$
where $\varphi_i$ are irreducible subrepresentations, by restricting $\varphi$ to a subspace $W_i$ with $F^m = W_1 \oplus \ldots \oplus W_n$. After choosing a basis for the spaces $W_i$ we can also see $\varphi_i: H \to \GL(m_i,\mathbb{F})$ with $m_i$ the dimension of $W_i$. 

If $\varphi$ is irreducible over any field extension of $\mathbb{F}$, then it is called absolutely irreducible and the field $\mathbb{F}$ is called a splitting field of $\varphi$.	By \cite[p. 292 \& p. 475]{curtis2006representation} we have the following result:
	\begin{thm} \label{thm_splitting_field}
		If a field $\mathbb{F}$ contains an $|H|$-th root of unity, then it is a splitting field of $\varphi$.
	\end{thm}

	Our main goal in this section is to give an upper bound for $\RF_G$ depending on how $\varphi$ splits over the complex numbers, as given in the following result. 
	
		\begin{thm} \label{thm_upper}
		Let $\varphi: H \to \GL(m, \Z)$ be an integral representation of a finite group $H$. Suppose that all the irreducible $\mathbb{C}$-subrepresentations have degree smaller or equal to $k$, then $$\RF_{\Z^m, \Inv(\varphi),B_{\Z^m}} \preceq \log^k.$$
	\end{thm}

The main tool here is the density of primes in certain subsets of $\N$. As often in number theory, we write $f \asymp g$ for functions $f,g: \R^+ \to \R^+$ if there exists constants $C_1,C_2 > 0$ such that
	$$ C_1 \leq \liminf_{x\to \infty} \frac{f(x)}{g(x)} \leq \limsup_{x\to \infty} \frac{f(x)}{g(x)} \leq C_2.$$
	In other words, $f \asymp g$ if and only if there exists a constant $M>0$ such that for all $x\geq M$ it holds that
	$$C_1g(x) \leq f(x) \leq C_2g(x).$$
	\begin{df} Let $S\subset \left\{p \in \mathbb{N}\mid p \text{ is prime} \right\}$ be any subset of primes. We define the function $\pi_S: \N \to \N$ as $\pi_S(x) = |\left\{p \in S\mid p\leq x \right\}|$, i.~e.~$\pi_S(x)$ is the number of elements in $S$ smaller than or equal to $x$. 
	\end{df}

The following results about the asymptotics of $\pi_S$ are well-known.
	\begin{lemma} \label{lem_PNT}
		The following are equivalent for subsets $S \subset \left\{p \in \mathbb{N}\mid p \text{ is prime} \right\}$ of prime numbers:
		\begin{enumerate}
			\item $\pi_S(x) \asymp \dfrac{x}{\log(x)}$,
			\item $\log(\prod_{p\in S, p\leq x}p) \asymp x$.
		\end{enumerate}
	\end{lemma}
	We know that various subsets $S$ satisfy the lemma above by certain prime number theorems:
	\begin{thm}
		The following subsets $S$ satisfy the lemma above:
		\begin{enumerate}
			\item the set $S = \left\{p \in \N \mid p \text{ is prime} \right\}$ of all primes;
			\item the set $S = \left\{p\mid p \text{ prime and } p\equiv a \mod n \right\}$ for fixed $a$ and $n$ with $\ggd(a,n) = 1$.
		\end{enumerate}
	\end{thm}
	\begin{proof}
		The first claim follows by the prime number theorem, the second by the prime number theorem in arithmetic progressions, see for example \cite{fine2007number}.
	\end{proof}
	It should be noted that this result is a special case of Chebotarev's density theorem \cite{tschebotareff1926bestimmung}. However, since we do not need this more general density theorem, we will not discuss it further in this paper. 
	
	For the sets $S$ mentioned above, the following result holds:
	\begin{prop} \label{prop_Dlogn_gPNT}
		Suppose $S$ is a subset of primes satisfying the equivalent conditions of Lemma \ref{lem_PNT}. There exist numbers $D_1,D_2>0$ such that for all $m\in\mathbb{N}$ there exists a prime number $p\in S$ with $p \leq D_1\cdot \log(m) + D_2$ with $p\nmid m$ (or equivalently $m\notin p\Z$).
	\end{prop}
	\begin{proof}
		Let all primes in this proof denote primes in $S$. We first show the result for all $m\geq 2$, where $D_2=0$.		Suppose such a fixed number $D_1$ does not exist. This means that for $n\in\mathbb{N}$, there must be at least one element $m_n \geq 2$ such that $p \mid m_n$ as soon as $p\leq n\cdot \log(m_n)$ and $p \in S$. As a consequence, we see that $$\prod_{\substack{p\leq n \log(m_n) \\ p \in S}} p \mid m_n.$$
		In particular, this means that $\displaystyle \prod_{\substack{p\leq n \log(m_n) \\ p \in S}}p \leq m_n$.\\
		We claim this is a contradiction to lemma \ref{lem_PNT}. Indeed, since $n \log(m_n)$ must go to infinity, this lemma says precisely that there exists a number $C$ such that for $n$ large enough we have the inequality
		$$0< C \leq \frac{\log \left( \prod_{\substack{p\leq n \log(m_n) \\ p \in S} }p \right)}{n \log(m_n)}.$$
		This means that $m_n^{Cn} \leq \displaystyle \prod_{\substack{p\leq n \log(m_n) \\ p \in S}}p $. Since $m_n \geq 2$ and for $n$ large $Cn > 1$, we have $m_n < m_n^{Cn}$. Hence, this lower bound on $\displaystyle \prod_{\substack{p\leq n \log(m_n) \\ p \in S}}p $ contradicts the earlier upper bound, so the statement must hold.
		
		If we pick $D_2$ to be the smallest prime in $S$, then the statement is also satisfied for $m = 1$.
	\end{proof}
	
Combining these results about the density of prime numbers with the theory of splitting fields for representations, we get the proof of Theorem \ref{thm_upper}.

	\begin{proof}[Proof of Theorem \ref{thm_upper}]
		Consider the standard word norm on $\Z^m$. Take $0\neq v \in B_{\Z^m}(r)$ arbitrary. Since this vector is non-zero, we can find one non-zero entry $a$. By Proposition \ref{prop_Dlogn_gPNT}, we can take a prime number $$p \leq D_1\cdot \log(|a|) + D_2 \leq D_1\cdot \log(\norm{v}_{\Z^m})+D_2$$ such that $a \notin p\Z$ and $p\equiv 1 \mod |H|$. In particular, $\Z_p$ contains a primitive $|H|$-th root of unity. By construction, $v\notin p\Z^m$. Let $\psi$ be the map $\Z^m \to \Z^m_p: v \mapsto v \mod p$. We claim that
		$$D_{\Z^m, \Inv(\varphi)}(v) \leq D_{\Z_p^m, \Inv(\phi)}(\psi(v))$$
		by Lemma \ref{lem_surjective}, where $\phi: H \to \GL(m,\Z_p)$ is the map $\varphi$ modulo $p$, so satisfying $$\phi(h)(\psi(v)) = \psi(\varphi(h)(v))$$ for all $v \in \Z^m, h \in H$. Indeed, we have to verify that $\psi^{-1}(\Inv(\phi)) \subset \Inv(\varphi)$, but this follows easily from the definition of $\phi$.
		
%		 Let $N\in \Inv(\phi)$. We need to show that $\varphi(h)(v) \in \psi^{-1}(N)$ for all $v\in \psi^{-1}(N)$ and $h\in H$. However, $\psi(\varphi(h)(v)) = \phi(h)(v\mod p)$. Since $v\mod p \in N$ and $N$ is $\phi$-invariant, $\phi(h)(v\mod p)$ lies in $N$. As a consequence, $\varphi(h)(v)$ lies in $\psi^{-1}(N)$. This shows the claim.  
		
		By theorem \ref{thm_splitting_field}, we know $\Z_p$ is a splitting field for $\phi$. Hence, we can write
		$$\phi = \phi_1\times \dots \times \phi_n,$$
		where $\phi_i$ is an absolutely irreducible subrepresentation corresponding to the $\phi$-invariant $\Z_p$-subspaces $V_i$ for each $1\leq i\leq n$. 
		
		Now we wish to apply Lemma \ref{lem_direct_sum}. Hence, we will argue that $\Inv(\phi)\cap V_i$ equals $\Inv(\phi_i)$. For $\Inv(\phi)\cap V_i \subset \Inv(\phi_i)$, suppose $N\in \Inv(\phi)$. If $w\in N\cap V_i$, then $\phi_i(h)(w) := \phi(h)(w) \in N\cap V_i$, since $N$ and $V_i$ are $\phi$-invariant. Hence, $N\cap V_i$ is $\phi_i$-invariant. The other inclusion is similar. It is also clear that the conditions
		$$\displaystyle \bigoplus_{i=1}^n\Inv(\phi_i) \subset \Inv(\phi)$$
		and $V_i \in \Inv(\phi_i)$ are satisfied.
		
	To apply the lemma, write $\psi(v) = v_1 + \dots + v_n$ where $v_i \in V_i$. Note that since $\psi(v)$ is non-zero, so is at least one of the $v_i$'s. We get
		$$D_{\Z_p^m, \Inv(\phi)}(\psi(v)) \leq \min\left\{D_{V_i, \Inv(\phi_i)}(v_i) \mid 1 \leq i \leq n \right\}.$$
		Let the dimension of $V_i$ as a $\Z_p$-vector space be $m_i$, then clearly $D_{V_i, \Inv(\phi_i),B_i}(v_i) \leq |V_i| \leq p^{m_i}$. In particular, let $k = \max\left\{m_i\mid 1\leq i\leq n \right\}$, then 
		$$D_{\Z_p^m, \Inv(\phi)}(\psi(v)) \leq p^k \leq (D_1\cdot \log(\norm{v}_{\Z^m})+D_2)^k.$$
		We end by noting that $k$ is also the maximal degree of the irreducible $\mathbb{C}$-subrepresentations of $\varphi$ by  \cite[Theorem 15.13]{isaacs2006character}.
	\end{proof}
	\section{Proof of the lower bound} \label{sec_lower}
	The goal of this section is to show that the inequality in Theorem \ref{thm_upper} is in fact optimal, by producing the corresponding lower bound. We proceed in three steps. First, we reduce our problem further to show that we may restrict our attention to $\Q$-irreducible representations. Then, we describe how matrices that commute with a given $\Q$-irreducible representation look like, by using the theory of Galois Descent. Finally, we proceed by comparing $\RF_{\Z^m, \Inv(\varphi), B_{\Z^m}}$ to the more restrictive $\RF_{\Z^m, \Com(\varphi), B_{\Z^m}}$, where $\Com(\varphi)$ consists of those subgroups that can be written as $\Im B$ for some matrix $B$ that commutes with $\varphi$. The claimed lower bound then follows from combining the previous steps.
	
	\subsection{Reduction to $\Q$-irreducible representations}
	Suppose $\varphi$ is an integral representation considered over the field $\Q$, then we can write 
	$$\varphi = \varphi_1\times \dots \times \varphi_n,$$
	where $\varphi_i$ with $1\leq i\leq n$ are irreducible $\mathbb{Q}$-subrepresentations with corresponding irreducible $\mathbb{Q}$-vector subspaces $W_i$, so $\varphi_i$ is the restriction of $\varphi$ to $W_i$. We write $P_i: \mathbb{Q}^m \to W_i$ for the natural projection onto $W_i$.

\begin{thm} \label{thm_irred_Q_split}
	Let $K_i$ denote $P_i(\Z^m)$, then $\RF_{\Z^m, \Inv(\varphi),B_{\Z^m}} \approx \max\left\{\RF_{K_i, \Inv(\varphi_i),B_i} \mid 1\leq i\leq n \right\}$.
\end{thm}

	By corollary \ref{cor_independent}, this equality holds for any choice of metric balls $B_i$, hence we do not specify the specific choice of metric on the $K_i$.

\begin{proof}
	Note first that the linear map $L(v) = \sum_{i=1}^nP_i(v)$ is the identity map on $\Q^m$ by construction. Set $K = \oplus_{i=1}^n K_i$. We have the following inclusion:
	$$\Z^m = L(\Z^m) \subset \oplus_{i=1}^n P_i(\Z^m) = \oplus_{i=1}^n K_i = K.$$
	As $K$ by construction is an abelian group of rank at most $m$, it must have rank exactly $m$ and contain $\Z^m$ as a subgroup of finite index. Since $\varphi$ is integral, we know that for all $v\in \Z^m$ the vector $\varphi(h)(v)$ lies in $\Z^m$. Decomposing $v$ as $\sum_{i=1}^n v_i$, where $v_i = P_i(v) \in K_i$, we see that
	$$\varphi(h)(v) = \sum_{i=1}^n\varphi(h)(v_i) = \sum_{i=1}^n\varphi_i(h)(v_i),$$ where thus $\varphi_i(h)(v_i) \in K_i$. In particular, if $v_i \in K_i$, then also $\varphi_i(h)(v_i) \in K_i$ and thus $\varphi_i: H \to \Aut(K_i)$ is well-defined. 
	Let $\bar{\varphi}$ be the map $$H\to \Aut(K): h \mapsto \sum_{i=1}^n\varphi_i(h)\circ P_i.$$ The previous computation shows that $\bar{\varphi}$ is the natural extension of $\varphi: H\to \GL(m, \Z)$ to $\GL(m, \Z) \subset \Aut(K)$. We can apply lemma \ref{lem_reduc_two} to obtain the equality
	$$\RF_{\Z^m, \Inv(\varphi),B_{\Z^m}} \approx \RF_{K, \Inv(\bar{\varphi}),B_K}.$$
	Indeed, $\Z^m \lhd_f K$, $\Inv(\bar{\varphi})\cap \Z^m = \Inv(\varphi)$, since $\bar{\varphi}$ is the natural extension of $\varphi$, and for every $N \in \Inv(\bar{\varphi})$ we have $[\Z^m: N\cap \Z^m] \leq [K:N]$.\\
	We conclude the proof by applying lemma \ref{lem_direct_sum}, using that $\bar{\varphi} = \varphi_1\times \dots \times \varphi_n$ on $K = \oplus_{i=1}^n K_i$.
\end{proof} 
	Note that $K_i \cong \Z^{m_i}$ for some $m_i$. Hence we may indeed reduce our attention to $\RF_{\Z^m, \Inv(\varphi), B_{\Z^m}}$, where $\varphi$ is $\mathbb{Q}$-irreducible.
	
	\subsection{Matrices commuting with $\mathbb{Q}$-irreducible representations} \label{sec_irrQrepr}
	Let $\varphi: H \to \GL(m, \Z)$ be an integral representation which is irreducible as a representation over $\mathbb{Q}$. In this section we discuss the matrices that commute with $\varphi$, i.e.~we investigate the $B\in \GL(m, \mathbb{Q})$ such that $\varphi(h)B = B\varphi(h)$ for all $h\in H$. For this, we will use some notions of Galois theory, for which we refer to \cite{Winter1974Fields} for more background and notation.
%	Suppose additionally that $L: \Z^m \to \Z^m$ is a $H$-equivalence between the $H$-modules $\mathbb{Q}^m$, i.e. a linear isomorphism that lets the diagram below commute.
%	\begin{equation*}
%		\begin{tikzcd}
%			\mathbb{Q}^m \arrow[rr, "L"] \arrow[dd, "\varphi(h)"] &  & \mathbb{Q}^m \arrow[dd, "\varphi(h)"] \\
%			&  &                               \\
%			\mathbb{Q}^m \arrow[rr, "L"]                          &  & \mathbb{Q}^m                         
%		\end{tikzcd}
%	\end{equation*}
%	What is the structure of $L$? 
%	If the map $L$ is represented by the matrix $B\in \GL(m, \mathbb{Q})$ with respect to the standard basis of $\mathbb{Q}^m$. Then so we look for the structure of a matrix $B$ that commutes with $\varphi$.\\

For the following definitions we work with the standard basis of $\mathbb{K}^m$. If $\sigma$ is an automorphism of $\mathbb{K}$ and $v\in \mathbb{K}^m$, then we write $\sigma(v)$ for the vector obtained by applying $\sigma$ on the entries of the vector $v$.
	\begin{df}
		We say a $\mathbb{K}$-vector subspace $W\subset \mathbb{K}^m$ is minimal over the field $\mathbb{F}$ if $W$ has a basis with entries over $\mathbb{F}$, but no basis with entries over any strictly smaller field $\mathbb{L}$.
	\end{df}
	Note that for any subspace $W \subset \mathbb{K}^m$, the set $\sigma(W) = \{\sigma(w) \mid w \in W\}$ is also a vector space over $\mathbb{K}$, and that $\sigma(W) = W$ if $W$ has a basis with entries over $\mathbb{F}$.
	\begin{lemma} \label{lem_sigmaW_irr}
		Let $\varphi: H\to \GL(m, \mathbb{Q})$ be a representation of a finite group and $\mathbb{K}$ be some number field. If $W$ is an irreducible $\mathbb{K}$-subspace of $\varphi^\mathbb{K}$ and $\sigma$ is an automorphism of $\mathbb{K}$, then $\sigma(W)$ is also $\varphi^\mathbb{K}$-irreducible. 
	\end{lemma}

	\begin{proof}
		Let $w\in W$. By definition of being $\varphi^\mathbb{K}$-invariant, we know that
		$$\forall h\in H: \varphi^\mathbb{K}(h)(w) \in W.$$
		Applying $\sigma$ to this expression and using that $\varphi^\mathbb{K}(h)$ is a rational matrix, we obtain
		$$\forall h \in H: \varphi^\mathbb{K}(h)(\sigma(w)) \in \sigma(W).$$
		This shows that $\sigma(W)$ is also $\varphi^\mathbb{K}$-invariant.
		
		Now suppose $\sigma(W)$ contains a strict subspace $W'$ that is invariant:
		$$\forall w\in W', \forall h \in H: \varphi^\mathbb{K}(h)(w) \in W'.$$
		Now apply $\sigma^{-1}$ to this expression to find that
		$$\forall w\in W', \forall h \in H: \varphi^\mathbb{K}(h)(\sigma^{-1}(w)) \in \sigma^{-1}(W').$$
		This shows that $\sigma^{-1}(W')$ is an invariant subspace. However, this is a strict subspace of $W$. Therefore, $\sigma^{-1}(W') = \{0\}$ and $W' = \{0\}$. This shows that $\sigma(W)$ is irreducible.
	\end{proof}
%	\begin{rem}
%		Using the setting above, if $W$ is irreducible and $\rho: H \to \Aut(W)$ is the corresponding map such that
%		$$\forall w \in W, \forall h \in H: \varphi^\mathbb{K}(h)(w) = \rho(h)(w),$$
%		then $\sigma\circ\rho\circ\sigma^{-1}$ is the map corresponding to $\sigma(W)$. Indeed, applying $\sigma$ to the expression above gives us
%		$$\forall w \in W, \forall h \in H: \varphi^\mathbb{K}(h)(\sigma(w)) = \sigma(\rho(h)(w)) = (\sigma\circ\rho\circ\sigma^{-1})(\sigma(w)).$$
%	\end{rem}
This leads to the following result describing how $\varphi$ splits over a splitting field.

	\begin{thm}
		\label{thm_notation}
		Let $\varphi: H\to \GL(m, \Z)$ be a $\mathbb{Q}$-irreducible representation of a finite group. Suppose it decomposes into absolutely irreducible components over a Galois extension $\mathbb{K}$ of $\mathbb{Q}$. Suppose that $W$ is an irreducible $\mathbb{K}$-subspace which is minimal over $\mathbb{F} \subset \mathbb{K}$. If $\sigma_1$ up to $\sigma_n$ denote the $n$ automorphisms of $\mathbb{K}$ distinct on $\mathbb{F}$, i.e. the automorphisms such that $\sigma_i\big|_\mathbb{F} \neq \sigma_j\big|_\mathbb{F}$, then $$\mathbb{K}^m = \bigoplus_{i=1}^n\sigma_i(W)$$ is a decomposition of $\mathbb{K}^m$ into irreducible subspaces.
		
	\end{thm}
	\begin{rem}
		Note that the automorphisms $\sigma_1$ to $\sigma_n$ can be identified with representatives of the cosets of $\Gal(\mathbb{K}/\mathbb{F})$ in $\Gal(\mathbb{K}/\mathbb{Q})$. Indeed, $\sigma_i(x) = \sigma_j(x)$ for all $x\in\mathbb{F}$ if and only if $\sigma_i\Gal(\mathbb{K}/\mathbb{F}) = \sigma_j\Gal(\mathbb{K}/\mathbb{F})$. As a consequence, if $\sigma$ is an automorphism of $\mathbb{K}$, then it induces a permutation on the cosets, implying that $\sigma\circ\sigma_i = \sigma_j\circ\sigma'$ for some $1\leq i,j\leq n$ and $\sigma'\in \Gal(\mathbb{K}/\mathbb{F})$. Since $W$ has a basis over $\mathbb{F}$, $\sigma'(W) = W$. Therefore, $\sigma$ induces a permutation on $\{\sigma_i(W)\mid 1\leq i\leq n\}$.
	\end{rem}
	\begin{proof}
		In \cite[theorem 3.1]{dekimpe2015existence} a similar statement is shown, from which we already know that if $\dim(W) = k$, then $m = kn$. However, the result does not specify the direct sum along which this decomposition holds, which is crucial for our purposes.
				
		Consider the subspace $V = \sigma_1(W) + \dots + \sigma_n(W)$. We will argue that this is in fact a direct sum with $V = \mathbb{K}^m$. 
		
		The $\mathbb{K}$-vector space $V$ is invariant under $\varphi^\mathbb{K}$, being the sum of invariant subspaces. Also, if $\sigma \in \Gal(\mathbb{K})$ and $\sum_{i=1}^n \sigma_i(w_i)$ is an arbitrary element of $V$, then
		$\sigma(\sum_{i=1}^n \sigma_i(w_i))$
		is still an element of $V$, as $\sigma$ permutes $\{\sigma_i(W)\mid 1\leq i \leq n\}$, so $\sigma(\sigma_i(w_i)) = \sigma_j(w'_i)$ for some $1\leq j \leq n$ and $w'_i \in W$. We conclude that $\sigma(V) = V$ for all $\sigma \in \Gal(\mathbb{K})$.
		
		Let $U$ denote the set $\{v \in V\mid \forall \sigma \in \Gal(\mathbb{K}): \sigma(v) = v\}$ of vectors in $V$ fixed under $\Gal(\mathbb{K})$. By the theory of Galois Descent, see \cite[Theorem 3.2.5]{Winter1974Fields}, $U$ is a $\mathbb{Q}$-vector subspace of $\mathbb{Q}^m$ for which $U\otimes_\mathbb{Q} \mathbb{K} = V$. (In particular, $V$ has a basis with entries over $\mathbb{Q}$.) However, $U$ is now a $\varphi^\mathbb{Q}$-invariant subspace. As $\varphi^\mathbb{Q}$ is irreducible, this implies that $U = \mathbb{Q}^m$ and $V = \mathbb{K}^m$.
		
		In conclusion, we obtain
		$$\mathbb{K}^m = \sigma_1(W) + \dots + \sigma_n(W).$$
		Thus, comparing the dimensions of both sides, this must be a direct sum.
	\end{proof}
Now let $B$ be a rational matrix that commutes with $\varphi: H \to \GL(m,\Z)$. Take $\mathbb{K}$ a number field such that $\varphi$ decomposes into absolutely irreducible components and $B$ can be put in its Jordan normal form over $\mathbb{K}$. We may assume that $\mathbb{K}$ is Galois over $\mathbb{Q}$.
	\begin{lemma} \label{lem_eigenvalue}
		Let $\varphi: H \to \GL(m, \Z)$ and let $B\in \GL(m, \mathbb{Q})$ commute with $\varphi$. There exists an absolutely irreducible $\mathbb{K}$-subspace $W$ of $\varphi^\mathbb{K}$ contained in an eigenspace of $B$. 
	\end{lemma}
	\begin{proof}
		By the Jordan decomposition of $B$ over $\mathbb{K}$, we know there exists at least one eigenvector $v$ for some eigenvalue $\lambda$. Consider 
		$$V =\text{span}_\mathbb{K}\{\varphi^\mathbb{K}(h)(v)\mid h \in H\}.$$
		This subspace is clearly $\varphi^\mathbb{K}$-invariant. Furthermore, we have
		$$B(\varphi^\mathbb{K}(h)v) = \varphi^\mathbb{K}(h)Bv = \lambda\varphi^\mathbb{K}(h)v$$
		for all basis vectors. Hence $V$ is in fact contained in the eigenspace of $\lambda$.\\
		Since $V$ is $\varphi^\mathbb{K}$-invariant, it contains an absolutely irreducible $\mathbb{K}$-subspace $W$. This ends the proof.
	\end{proof}

		Take the $\mathbb{K}$-vector subspace $W \subset \mathbb{K}^m$ as in the previous lemma. Suppose it is minimal over the field $\mathbb{F}$. Then $\lambda \in \mathbb{F}$, and for all $\sigma_i$ in the direct sum $\mathbb{K}^m = \oplus_{i=1}^n \sigma_i(W)$, we have
		$$\forall w\in W: B\sigma_i(w) = \sigma_i(Bw) = \sigma_i(\lambda w) = \sigma_i(\lambda)\sigma_i(w).$$ 

	We get the following result:
	\begin{prop} \label{prop_irr_Q_basis}
		Using the notation as above, if $B$ commutes with a $\mathbb{Q}$-irreducible representation $\varphi$, then it is of the form
		\begin{equation*}
			B \sim_\mathbb{K} \begin{pmatrix}  \sigma_1(\lambda)\mathbb{1}_k & 0 & \dots & 0\\ 0 & \sigma_2(\lambda)\mathbb{1}_k & &\\ \vdots & & \ddots &\\0 & 0 & \dots & \sigma_n(\lambda)\mathbb{1}_k \end{pmatrix}
		\end{equation*}
		with respect to a basis along the direct sum $\oplus_{i=1}^n \sigma_i(W)$.
	\end{prop}
	\begin{ex} \label{ex_commuting}
		The quaternion group $Q = \{\pm1, \pm i, \pm j, \pm k\}$ of order $8$ has a faithful group representation $\varphi: Q \to \GL(4, \Z)$ given by
		\begin{align*} \varphi(i) &= \begin{pmatrix}0&1&0&0\\-1&0&0&0\\0&0&0&-1\\0&0&1&0\end{pmatrix}, \\ \varphi(j) &= \begin{pmatrix}0&0&1&0\\0&0&0&1\\-1&0&0&0\\0&-1&0&0\end{pmatrix}, \\ \varphi(k) &= \begin{pmatrix}0&0&0&1\\0&0&-1&0\\0&1&0&0\\-1&0&0&0\end{pmatrix}.		\end{align*}
		The following matrix commutes with every element of $\varphi(Q)$:
		$$ B = \begin{pmatrix}1&-1&-2&0\\1&1&0&2\\2&0&1&-1\\0&-2&1&1\end{pmatrix}.$$
		A minimal splitting field for both $B$ and $\varphi$ is $K = \mathbb{Q}(\sqrt{5}i)$ with $\Gal(K,\Q) = \{\sigma_1, \sigma_2\}$ and $\sigma_1(x) = x$ and $\sigma_2(x) = \bar{x}$, where $\bar{x}$ denotes the complex conjugation. With respect to the basis given by $\left\{\sigma_1(w_1), \sigma_1(w_2), \sigma_2(w_1), \sigma_2(w_2) \right\}$ with
		$$ w_1 = \begin{pmatrix} -\sqrt{5}i\\1\\2\\0\end{pmatrix}\text{ and }w_2 = \begin{pmatrix} 1\\\sqrt{5}i\\0\\2\end{pmatrix}$$
		we obtain the following matrices:
		$$B \sim_\mathbb{C} \begin{pmatrix}1-\sqrt{5}i&0&0&0\\0&1-\sqrt{5}i&0&0\\0&0&1+\sqrt{5}i&0\\0&0&0&1+\sqrt{5}i\end{pmatrix},$$
		$$\varphi(i), \varphi(j), \varphi(k) \sim_\mathbb{C} \begin{pmatrix}0&-1&0&0\\1&0&0&0\\0&0&0&1\\0&0&-1&0\end{pmatrix}, \, \begin{pmatrix}\frac{\sqrt{5}i}{2}&-\frac{1}{2}&0&0\\-\frac{1}{2}&-\frac{\sqrt{5}i}{2}&0&0\\0&0&-\frac{\sqrt{5}i}{2}&-\frac{1}{2}\\0&0&-\frac{1}{2}&\frac{\sqrt{5}i}{2}\end{pmatrix}, \,  \begin{pmatrix}\frac{1}{2}&\frac{\sqrt{5}i}{2}&0&0\\ \frac{\sqrt{5}i}{2}&-\frac{1}{2}&0&0\\0&0&\frac{1}{2}&-\frac{\sqrt{5}i}{2}\\0&0&-\frac{\sqrt{5}i}{2}&-\frac{1}{2}\end{pmatrix}.
$$
		
	\end{ex}
	\begin{lemma} \label{lem_copy_mat}
	
		Suppose $B$ is an integral matrix of full rank that commutes with the $\mathbb{Q}$-irreducible representation $\varphi$ of dimension $m$. If we take notations as in Theorem \ref{thm_notation} and write $x = \prod_{1\leq i\leq n} \sigma_i(\lambda)$, then $x \in \Z$ is an integer, $\det B = x^k$ with $k$ the dimension of a $\C$-irreducible subspace of $\varphi$ and
		$$x \Z^m \subset \Im B.$$
	\end{lemma}
	\begin{proof}
		Let $f(X)$ be the polynomial
		$$ \prod_{1\leq i \leq n} (X - \sigma_i(\lambda)) = \sum_{i=0}^n a_i X^i.$$
		If $\mathbb{K}$ is the Galois closure of the field $\mathbb{F}$, then we see that all $\sigma \in \Gal(\mathbb{K}/\mathbb{Q})$ permute the roots of $f(X)$. In particular, we see that $\sigma(f(X)) = f(X)$, so $f(X)$ is a rational polynomial. In fact, it is integral, since all roots are algebraic numbers, because they are roots of the characteristic polynomial of $B$, which is $f(X)^k$. The number $x$ equals $a_0$. Note also that $a_n$ equals $1$.
		
		It is left to prove that $x \Z^m \subset \Im B$. We will argue that the largest possible order of an element in the quotient group $\Z^m/\Im B$ is divisible by $x$. The largest possible order is precisely the largest invariant factor of the Smith-normal form decomposition of $B$. At the other hand, it is known that this equals $(\det A)/D$, where $D$ is the greatest common divisor of the minors. Hence, we need to prove that all minors are multiples of $x^{k-1}$.
	
		Recall that the minors are the entries of the adjugate matrix $\text{Adj}(B)$.	Set $M = \sum_{i=1}^n a_i B^{i-1}$. This is clearly an integral matrix. We find
		$$BM = MB \sim_\mathbb{C} \begin{pmatrix}  (\sum_{i=1}^n a_i \sigma_1(\lambda)^{i-1})\sigma_1(\lambda)\mathbb{1}_k & 0 & \dots \\ 0 & (\sum_{i=1}^n a_i \sigma_2(\lambda)^{i-1})\sigma_2(\lambda)\mathbb{1}_k & \dots \\ \vdots & \vdots & \ddots \end{pmatrix}.$$ 
		Note that $(\sum_{i=1}^n a_i \sigma_1(\lambda)^{i-1})\sigma_1(\lambda) = \sum_{i=1}^n a_i \sigma_1(\lambda)^{i} = -a_0 = -x$, since $\sigma_1(\lambda)$ is a root of $f(X)$.	We have thus found that
		$$ BM = MB = -x \mathbb{1},$$
		so $B^{-1} = -\frac{1}{x}M$ and hence
		$$ \text{Adj}(B) = (\det B)B^{-1} = x^k (-\frac{1}{x}M) = -x^{k-1}M.$$
		This ends the proof. 
	\end{proof}
	
	\subsection{Reduction to Commuting Matrices}
Any finite index subgroup of $\Z^m$ is equal to $\Im B := B(\Z^m)$ for some matrix $B \in \Z^{m\times m}$ which is invertible over $\Q$, i.e.~$\det(B) \neq 0$. If $\varphi: H \to \GL(m,\Z)$ is an integer representation, then the set $\Inv(\varphi)$ consists of the images $\Im B$ such that $\Im \varphi(h)B = \Im B$ for all $h \in H$. In the special case when $B$ commutes with $\varphi(h)$, then surely $\Im \varphi(h)B = \Im B\varphi(h) = \Im B$, because $\varphi(h) \in \GL(m, \Z)$. Hence, commuting gives a stronger, yet not equivalent, notion than invariant subgroups.
	
	The goal of this section is to show that we may replace $\Inv(\varphi)$ by those subgroups which come from commuting matrices. In the proof of the lower bound, we will use this in order to apply lemma \ref{lem_copy_mat}.
	\begin{df}
		Let $\Com(\varphi)$ be the set $\left\{\Im B \in \Inv(\varphi)\mid \forall h\in H: B\varphi(h)=\varphi(h)B \right\}$.
	\end{df}
	\begin{ex}
		Continuing on Example \ref{ex_commuting}, the matrix $B$ induces a subgroup $\Im B$ in $\Com(\varphi)$, where $\varphi$ is the action of the quaternions. Note that applying elementary column operations on $B$ (over the ring $\Z$) does not change the subgroup $\Im B$. However, the new matrix does not have to commute with $\varphi$ anymore.
		
A bit more work shows that the corresponding subgroups $\Com(\varphi)$ and $\Inv(\varphi)$ are also different here. Indeed, consider
		$$\Im \begin{pmatrix} 1&0&0&0\\0&1&0&0\\0&0&1&0\\1&1&1&2 \end{pmatrix}.$$
		One can verify that this subgroup is invariant under $\varphi$. However, the determinant of the matrix is two, and so it does not lie in $\Com(\varphi)$, as lemma \ref{lem_copy_mat} shows that the determinant should be a proper square then.  
	\end{ex}
	\begin{lemma}
Let $\varphi: H \to \GL(m,\Z)$ be an integral representation and $\Im B$ is $\varphi$-invariant subgroup with $B \in \Z^{m \times m}$ and $\det B \neq 0$. Then there exists an integral representation $\bar{\varphi}: H \to \GL(m, \Z)$ such that for all $h\in H$: $\varphi(h)B = B\bar{\varphi}(h)$. In particular, $\varphi$ and $\bar{\varphi}$ are $\mathbb{Q}$-equivalent integral representations.
	\end{lemma}
	\begin{proof}
		Define $\bar{\varphi}(h) = B^{-1}\varphi(h)B$. This clearly defines an $\mathbb{Q}$-equivalent representation. It suffices to argue that $\bar{\varphi}$ is an integral representation.\\
		Take $h\in H$ arbitrary. We argue that $\bar{\varphi}(h)$ is integral. Since $\Im B$ is $\varphi$-invariant, we know that for all standard vectors $e_i$, we find an integral vector $u_i$ such that 
		$$\varphi(h) Be_i = Bu_i.$$
		Note that $Be_i$ is the $i$-th column of $B$, say $B_i$. In other words, we have $\varphi(h)B_i = Bu_i$. As a consequence, we have
		$$\varphi(h) B = B \begin{pmatrix}u_1&\dots&u_m\end{pmatrix},$$
		so $\bar{\varphi}(h) = \begin{pmatrix}u_1&\dots&u_m\end{pmatrix} \in \Z^{m\times m}$.
	\end{proof}
	By \cite[p. 559]{curtis2006representation}, we have the following result:
	\begin{thm}[(Corollary of) Jordan-Zassenhaus Theorem]
		Given an integral representation $\varphi: H \to \GL(m, \Z)$. The set $\left\{\tilde{\varphi}: H \to \GL(m, \Z) \mid \varphi \sim_\mathbb{Q} \tilde{\varphi} \right\}$ of all integral representations $\mathbb{Q}$-equivalent to $\varphi$ splits in a finite number of equivalence classes under $\Z$-equivalence.
	\end{thm}
	We use this to prove the following:
	\begin{prop}
		\label{prop_commut}
		There exists a constant $M>0$ such that for each $B \in \Z^{m \times m}$ with $\det(B) \neq 0$ for which $\Im B \in \Inv(\varphi)$, there exists a matrix $B'\in \Com(\varphi)$ such that $\Im B' \leq \Im B$ and $0< |\det B'| \leq M |\det B|$.
	\end{prop}
	\begin{proof}
		By the Jordan-Zassenhaus Theorem, we find a finite number of representatives of the $\Z$-equivalence classes of the set $\left\{\tilde{\varphi}: H \to \GL(m, \Z) \mid \varphi \sim_\mathbb{Q} \tilde{\varphi} \right\}$, say $\varphi_1$ up to $\varphi_k$. We also fix the matrices $A_i \in \GL(m, \mathbb{Q})$ such that $\varphi_i(h)A_i = A_i\varphi(h)$. By multiplying the matrices $A_i$ by a well-chosen scalar matrix, we may assume that $A_i \in \Z^{m\times m}$.\\
		Take $M = \max\left\{|\det A_i|\mid 1\leq i\leq k \right\}$. Let $B$ be an arbitrary matrix for which $\Im B$ is $\varphi$-invariant. By the previous lemma, we know that there is an integral representation $\bar{\varphi}$ such that $\varphi(h)B = B\bar{\varphi}(h)$. Also, we know that $\bar{\varphi}$ lies in the $\Z$-equivalence class of some $\varphi_i$. Hence, there exists a matrix $P\in \GL(m, \Z)$ such that $\bar{\varphi}(h)P = P\varphi_i(h)$.
		We claim that the matrix $BPA_i$ is the matrix $B'$ from the proposition's statement.\\
		First, by construction $PA_i$ is an integral matrix, hence $\Im BPA_i$ is clearly a subgroup of $\Im B$. Its determinant can be estimated by
 		$$0 < |\det (BPA_i)| = |\det A_i|\cdot|\det B| \leq  M|\det B|.$$
		It is left to show that this matrix commutes with $\varphi(h)$ for every $h\in H$. We get
		$$\varphi(h)BPA_i = B\bar{\varphi}(h)PA_i = BP\varphi_i(h)A_i = BPA_i \varphi(h).$$
	\end{proof}
	\begin{thm} \label{thm_reduc_to_comm}
		The following equality holds:
		$$\RF_{\Z^m, \Inv(\varphi),B_{\Z^m}} \approx \RF_{\Z^m, \Com(\varphi),B_{\Z^m}}.$$
	\end{thm}
	\begin{proof}
		This follows directly from Proposition \ref{prop_commut} and Lemma \ref{lem_change_P}.
	\end{proof}

	\subsection{Proof of the Main Result}
	In this last subsection we will prove the claimed lower bound, and we will prove the main result, i.e. Theorem \ref{thm_main}.
	\begin{thm} \label{thm_lower}
		If $\varphi: H \to \GL(m, \Z)$ is $\mathbb{Q}$-irreducible and has an absolutely irreducible subspace of dimension $k$ over $\mathbb{C}$, then $\RF_{\Z^m, \Com(\varphi),B_{\Z^m}} \succeq \log^k$.
	\end{thm}
	\begin{proof}
		Take the first standard basis vector $v$ of $\Z^m$ and consider the elements $$v_s = \lcm(1,2, \dots , s)v \in B_{\Z^m}(\lcm(1,2, \dots , s))$$ for any $s \in \N$. We first show that $D_{\Z^m, \Com(\varphi)}(v_s) \geq s^k$. Indeed, take any $\Im B \in \Com(\varphi)$ such that $v_s \notin \Im B$. By lemma \ref{lem_copy_mat}, we know that $|\det B| = x^k$ for some $x\in \Z$ and $x\Z^m \subset \Im B$. Thus $v_s$ does also not lie in $x\Z^m$. As $v_s \in l\Z^m$ for any $1 \leq l \leq s$, we conclude that $s < x$, implying that $|\det B| \geq s^k$.
		
		Now, by the prime number theorem, this implies that $D_{\Z^m, \Com(\varphi)}(v_s) \geq (D\cdot \log(\lcm(1,2, \dots , s)))^k$ for some fixed constant $D>0$, so
		$$ \RF_{\Z^m, \Com(\varphi),B_{\Z^m}}(\lcm(1,2, \dots , s)) \geq (D\cdot \log(\lcm(1,2, \dots , s)))^k.$$ This ends the proof.	
	\end{proof}
	All the previous results show that Theorem \ref{thm_main} holds:
	\begin{proof}[Proof of Theorem \ref{thm_main}]
		The upper bound follows by the equalities
		$$\RF_{G, \nu(G), B_G} \approx \RF_{K, \Inv(\bar{\varphi}),B_{K}} \preceq \log^k$$
		by Theorems \ref{thm_finite_extension} and \ref{thm_upper}.\\
		The lower bound follows by the equalities
		\begin{equation*}
			\begin{split}
				\RF_{G, \nu(G), B_G} & \approx \RF_{K, \Inv(\varphi),B_{K}}\\
				& \approx \max\left\{\RF_{K_i, \Inv(\varphi_i),B_i} \mid 1\leq i\leq n \right\} \\
				& \approx \max\left\{\RF_{K_i, \Com(\varphi_i),B_i} \mid 1\leq i\leq n \right\}\\
				& \succeq \max\left\{ \log^{k_i} \mid 1\leq i\leq n \right\}\\
				& \approx \log^k,			\end{split}
		\end{equation*}
	where $k_i$ is the dimension of the $\mathbb{C}$-irreducible subrepresentations of $\varphi_i$. Here we used consecutively Theorems \ref{thm_finite_extension} and \ref{thm_irred_Q_split}, Proposition \ref{thm_reduc_to_comm} and Theorem \ref{thm_lower}.
	\end{proof}

\section{Applications and examples} \label{sec_applic}
	If the maximal torsion-free abelian subgroup of a virtually abelian group $G$ has rank $m$, then Theorem \ref{thm_main} says that $\RF_G$ equals $\log^k$ for some $1\leq k \leq m$. The following corollary tells us that this result is optimal, namely here exists such a group for every $k$ between $1$ and $m$. 
	\begin{cor}
		For every $m\geq 1$ and every $1 \leq k \leq m$, there exists a semidirect product $G = \Z^m \rtimes_\varphi H$ for a finite group $H$, such that $\RF_G$ equals $\log^k$. In particular, for every $k\in \mathbb{N}$, there exists a finitely generated, residually finite group $G$ such that $\RF_G$ equals $\log^k$.
	\end{cor}
	\begin{proof}
		Consider the permutation representation of the symmetric group $S_{k+1}$ on $\Z^{k+1}$. It is known that it decomposes into two absolutely irreducible representations over $\mathbb{Q}$ corresponding to the subspaces $\left\{(l,l, \dots, l) \in \mathbb{Q}^{k+1}\mid l\in \mathbb{Q} \right\}$ and $\left\{(l_1,l_2, \dots ,l_{k+1})\mid \sum_{i=1}^{k+1}l_i = 0
\right\}$, see for example \cite[exercise 2.6]{serre1977linear}. Therefore, there exists an absolutely irreducible representation $\varphi_k: S_{k+1} \to \GL(k, \Z)$.	The result follows by taking $G$ to be $\Z^m\rtimes_\varphi S_{k+1}$, where $\varphi = \varphi_k \times \Id_{m-k}$.
	\end{proof}

Since virtually abelian groups are quasi-isometric if and only if their maximal torsion-free abelian subgroups have the same rank, the following result is now evident:
\begin{cor}
	\label{cor:QI}
	Residual finiteness growth is not a quasi-isometric invariant in the class of virtually abelian groups. In fact, if $G$ is a virtually abelian group with torsion-free abelian subgroup of rank at least two, then there exists a group $G^\prime$ quasi-isometric to $G$ with $\RF_G \not \approx \RF_{G^\prime}$.
\end{cor}

We also wish to remark that the power $k$ can be effectively computed from the character table of the finite group $H$. To illustrate this point, let us recall some notions. For further details, we refer to \cite[Chapter 2]{isaacs2006character}.
\begin{df}
	Let $\varphi: H \to \GL(m, \mathbb{C})$ be a complex representation of a finite group $H$. The character afforded by $\varphi$, denoted by $\chi_\varphi$, is the function \begin{align*}\chi_\varphi: H &\to \mathbb{C}\\ h &\mapsto \text{Trace}(\varphi(h)).\end{align*}
\end{df}
Here $\varphi(h)$ is represented as a matrix with respect to any basis of $\mathbb{C}^m$. Note that $\chi_\varphi(e_H)$ must equal the dimension $m$ of $\varphi$. A finite group $H$ allows only a finite number of non-equivalent irreducible representations, say $\varphi_1$ to $\varphi_r$. Let $\chi_1$ to $\chi_r$ denote the corresponding characters, called irreducible characters. 

Suppose $\zeta_1$ and $\zeta_2$ are two characters of $H$, then we can define the following inner product:
\begin{df}
	$$\langle \zeta_1, \zeta_2\rangle = \frac{1}{\vert H \vert} \sum_{h\in H} \zeta_1(h)\overline{\zeta_2(h)} \in \mathbb{C}.$$
\end{df}
With respect to this inner product, the irreducible characters arise as an orthonormal basis in the sense that $\langle \chi_i, \chi_j \rangle = \delta_{i,j}$, where $\delta_{i,j}$ is the Kronecker delta, and if $\chi_\varphi$ is any character, then
$$\chi_\varphi = \sum_{i=1}^r\langle \chi_\varphi , \chi_i \rangle \chi_i.$$
In terms of the original representation, this means $\varphi_i$ arises $\langle \chi_\varphi, \chi_i\rangle$ times as a subrepresentation of $\varphi$.

We obtain the following theorem:
\begin{thm} \label{thm_character_table}
	Let $\varphi: H \to \GL(m, \Z)$ be a representation of a finite group $H$. Then $\RF_{\Z^m, \Inv(\varphi), B_{\Z^m}}$ equals $\log^k$ with $k = \max\left\{\chi_i(e_H)\mid 1\leq i\leq r: \langle \chi_\varphi, \chi_i \rangle \neq 0
	\right\}$, where $\chi_1$ to $\chi_r$ are the irreducible characters of $H$.
\end{thm}
\begin{proof}
	If $\langle \chi_\varphi, \chi_i \rangle \neq 0$, then the irreducible representation $\varphi_i$ is a subrepresentation of $\varphi$. Its dimension is $\chi_i(e_H)$. Hence, by Theorem \ref{thm_main}, the maximum of the set $\left\{\chi_i(e_H)\mid 1\leq i\leq r: \langle \chi_\varphi, \chi_i \rangle \neq 0
	\right\}$ is precisely the power of the logarithm.
\end{proof}

If $h_1$ and $h_2$ are conjugate and $\chi_\varphi$ is a character, then $\chi_\varphi(h_1) = \chi_\varphi(h_2)$. Therefore, characters are fully determined by the values on their conjugacy classes $C_1$ to $C_t$.
\begin{df}
	The character table of a finite group $H$ is the matrix $(\chi_i(C_j))_{1\leq i\leq r, 1 \leq j \leq t}$ consisting of the value of every irreducible character per conjugacy class.
\end{df}
Usually, the characters are ordered according to increasing dimension. 

\begin{ex} \label{ex_D4}
	The character table of $D_4 = \{a, b\mid a^4=b^2=1, bab = a^{-1}\}$, the dihedral group on $8$ elements, is given in Table \ref{table_D4}. There are five irreducible characters.
	
	Suppose we are given the representation $\varphi: D_4 \to \GL(m, \Z)$ with \begin{align*}
		\varphi(a) &= \begin{pmatrix}-1&1&-1\\-2&0&-1\\2&-1&2 \end{pmatrix} \\ \varphi(b) &= \begin{pmatrix}-3&0&-2\\0&1&0\\4&0&3\end{pmatrix}.
	\end{align*}
	The values of its character $\chi_\varphi$ have been added in Table \ref{table_D4}. Using those values, we can compute $\langle \chi_\varphi, \chi_5\rangle$:
	$$\langle \chi_\varphi, \chi_5\rangle = \frac{1}{8}\left( 1\cdot (3\cdot 2) + 2\cdot (1\cdot 0) + 1\cdot (-1\cdot (-2)) + 2\cdot (1\cdot 0) + 2\cdot (1 \cdot 0)\right) = 1.$$
	Hence, the two-dimensional irreducible representation corresponding to $\chi_5$ occurs in the decomposition of $\varphi$. Therefore, $\RF_{\Z^m, \Inv(\varphi), B_{\Z^m}}$ equals $\log^2$. (In fact, one can verify that 
	$\chi_\varphi = \chi_1 + \chi_5$.)
	\begin{table}
		\centering
	\begin{tabular}{r | c c c c c}
		& $\{1\}$ & $\{a, a^3\}$ & $\{a^2\}$ & $\{ab, a^3b\}$ & $\{b, a^2b\}$ \\ \hline
		$\chi_1$ & 1 & 1 & 1&1&1\\
		$\chi_2$ & 1 & -1 & 1 & -1 & 1 \\
		$\chi_3$ & 1 & 1 & 1 & -1 & -1\\
		$\chi_4$ & 1 & -1&1 & 1 & -1 \\
		$\chi_5$ & 2 & 0 & -2 & 0 & 0 \\ \hline
		$\chi_\varphi$ & 3 & 1 & -1 & 1 & 1
	\end{tabular}
	\caption{The character table of $D_4$ and the values of the character $\chi_\varphi$ of Example \ref{ex_D4}}.
	\label{table_D4}
	\end{table}
\end{ex}

\begin{cor}
	Let $\varphi: H \to \GL(m, \Z)$ be a representation of a finite group $H$. Then $\RF_{\Z^m, \Inv(\varphi), B_{\Z^m}}$ equals $\log$ if and only if $\varphi(H) \leq \GL(m, \Z)$ is an abelian group.
\end{cor}
\begin{proof}
It is clear that $\Inv(\varphi)$ equals $\Inv(\psi)$ with $\psi: \varphi(H) \to \GL(m, \Z)$ the natural inclusion map. Hence, $\RF_{\Z^m, \Inv(\varphi), B_{\Z^m}}$ equals $\RF_{\Z^m, \Inv(\psi), B_{\Z^m}}$.
	
	By \cite[Corollary 2.6]{isaacs2006character}, a finite group is abelian if and only if all its irreducible characters are one-dimensional. Thus, if $\varphi(H)$ is abelian, then its irreducible characters are one-dimensional, and therefore theorem \ref{thm_character_table} tells that its residual finiteness growth is $\log$. Conversely, if $\RF_{\Z^m, \Inv(\psi), B_{\Z^m}}$ is $\log$, then the irreducible subrepresentations of $\psi^\C$ are one-dimensional and thus $\varphi(H)$ must be abelian, as it is conjugate to a subgroup of diagonal matrices.
\end{proof}

	\section{Open questions} \label{sec_open_questions}
Our main result, namely Theorem \ref{thm_main}, describes the residual finiteness growth $\RF_G$ for finitely generated virtually abelian groups $G$. Hence, it forms the first step in determining $\RF_G$ for all finitely generated virtually nilpotent groups, a question following from previous work in \cite{bou2011asymptotic} and which is still open in full generality, see \cite{survey2022}. However, the exact nature of our work raises related problems in determining the residual finiteness growth.

As mentioned before in Corollary \ref{cor:QI}, we concluded that $\RF_G$ is not a quasi-isometric invariant for virtually abelian groups $G$ in a very strong sense, although it (trivially) is one in the class of abelian groups. It is currently unknown whether the residual finiteness growth is a quasi-isometric invariant for finitely generated nilpotent groups. 

\begin{ques}
	Let $G_1$ and $G_2$ be finitely generated nilpotent groups that are quasi-isometric. Does it hold that $\RF_{G_1} \approx \RF_{G_2}$?
\end{ques}

Another interesting fact following from our main result is that $\RF_G$ only depends on properties of the representation over the complex numbers. Recall that every finitely generated torsion-free nilpotent group $G$ has a unique Mal'cev completion $G^\mathbb{F}$ over any field $\mathbb{F}$ of characteristic zero (see \cite{sega83-1} for more details). Motivated by the virtually abelian groups, we wonder whether $\RF_G$ only depends on $G^\C$.

\begin{ques}
Let $G_1$ and $G_2$ be finitely generated torsion-free nilpotent groups with isomorphic complex Mal'cev completions $G_1^\C \cong G_2^\C$. Does it hold that $\RF_{G_1} \approx \RF_{G_2}$? 
\end{ques}
Both questions are related in the following sense. If $G_1$ and $G_2$ are quasi-isometric nilpotent groups, then it is conjectured that their real Mal'cev completions $G_1^{\R}$ and $G_2^{\R}$ are isomorphic (it is even conjectured that the converse holds as well) and in particular their complex Mal'cev completions as well. Hence the latter question is a stronger version of the first one if the conjecture holds.

Even more generally, we could study the same question for virtually nilpotent groups, where we also have a representation of a finite group $H \to \Aut(G^\C)$ into the automorphisms of the complex Mal'cev completion, for example as described in the rational case in \cite[Section 4.1.]{dere22-1}. 

\bibliographystyle{plain}
\bibliography{Dere_Matthys_virtually_abelian}
\end{document}